\documentclass[12pt]{article}

\usepackage{amsfonts}
\usepackage{latexsym}
\usepackage{amssymb}
\usepackage{subfigure}
\usepackage[T1]{fontenc}
\usepackage[latin1]{inputenc}
\usepackage{graphicx}
\usepackage{color}
\usepackage{amsmath}
\usepackage{amsthm}
\newtheorem{notation}{Notations}[section]
\newtheorem{remarque}{Remark}{\empty{}}
\newtheorem{thm}[notation]{Theorem}

\newtheorem{defin}[notation]{Definition} 

\newtheorem{prop}[notation]{Proposition}
\newtheorem{lem}[notation]{Lemma}

\newcommand{\intpartialD}{\int_{\partial \Omega_D}}
\newcommand{\intomega}{\int_\Omega}
\newcommand{\dive}{\operatorname{div}}
\newcommand{\ind}{I_{B(0,1)}}

\newcommand{\restD}{_{\displaystyle \left|_{\partial \Omega_D}\right.}}
\newcommand{\gradh}{\nabla^h}
\newcommand{\divh}{\operatorname{div}^h}
\newcommand{\dom}{\operatorname{dom}}
\title{Continuous Primal-Dual methods for Image Processing}
\author{ M. Goldman\footnote{CMAP, CNRS UMR 7641, Ecole Polytechnique, 91128 Palaiseau Cedex, France.
e-mail:  goldman@cmap.polytechnique.fr}}

\begin{document}

\maketitle

\begin{abstract}
In this article we study a continuous Primal-Dual method proposed by Appleton and Talbot and generalize it to other problems in image processing. We interpret it as an Arrow-Hurwicz method which leads to a better description of the system of PDEs obtained. We show existence and uniqueness of  solutions and get a convergence result for the denoising problem. Our analysis also yields new \textit{a posteriori} estimates.
\end{abstract}

\textbf{Acknowledgements.} I would like to warmly thank my PhD advisor Antonin Chambolle for suggesting me this problem and for our fruitful discussions. This research was partially supported by ANR project MICA (2006-2009).

\section{Introduction}
In imaging, duality has been recognized as a fundamental ingredient for designing numerical schemes solving variational problems involving a total variation term. Primal-Dual methods were introduced in the field by  Chan, Golub and Mulet in \cite{changolubmulet}. Afterwards,  Chan and Zhu \cite{zhuchan} proposed to rewrite the discrete minimization problem as a min-max and  solve it using an Arrow-Hurwicz \cite{arrowhurwicz} algorithm which is a gradient ascent in one direction and a gradient descent in the other. Just as for the simple gradient descent, one can think of extending this method to the continuous framework. This is in fact what  does the algorithm previously proposed by Appleton and Talbot in \cite{AT} derived by analogy with discrete graph cuts techniques. The first to notice the link between their method and Primal-Dual schemes were Chambolle and al. in \cite{chambollecremerspock}.\\
Besides its intrinsic theoretical interest, considering the continuous framework has also pratical motivations. Indeed, as illustrated by  Appleton and Talbot in \cite{AT}, this approach leads to higher quality results compared with fully discrete schemes such as those proposed by Chan and Zhu.  We will numerically illustrate this in the final part of this paper.\\

 This paper proposes to study the continuous Primal-Dual algorithm following the philosophy of the work done  for the gradient flow by Caselles and its collaborators (see the book of Andreu and al. \cite{andreu} and the references therein). We give a rigorous definition of the system of PDEs which is obtained and show existence and uniqueness of a solution to the Cauchy problem. We  prove strong $L^2$ convergence to the minimizer for the Rudin-Osher-Fatemi model and derive some \textit{a posteriori} estimates. As a byproduct of our analysis we also obtain \textit{a posteriori} estimates for the numerical scheme proposed by Chan and Zhu.\\

\subsection{Presentation of the problem}
Many problems in image processing can be seen as minimizing in $BV\cap L^2$ an energy of the form
\begin{equation}\label{generalprob}J(u)=\intomega |Du|+G(u)+\intpartialD |u-\varphi| \end{equation}

The notation $\displaystyle \intomega |Du| $ stands for the total variation of the function $u$ and is rigourously defined in Definition \ref{TV}. We assume that $\Omega$ is a bounded Lipschitz open set  of $\mathbb{R}^d$ (in applications for image processing, usually $d=2$ or $d=3$) and that $\partial \Omega_D$ is a subset of $\partial \Omega$. The function $\varphi$ being given in $L^1(\partial \Omega_D)$, the term $\displaystyle \intpartialD |u-\varphi|$ is a Dirichlet condition on $\partial \Omega_D$. We  call $\partial \Omega_N$ the complement of $\partial \Omega_D$ in $\partial \Omega$ and  assume that $G$ is convex and continuous  in $L^2$ with
\[G(u)\leq C(1+|u|^p_2) \qquad \textrm{with } 1\leq p\leq +\infty\]

In this paper we note $|u|_2$  the $L^2$ norm of $u$. According to Giaquinta and al. \cite{giaquinta} we have,

\begin{prop}
The functional $J$ is convex and lower-semi-continuous (lsc) in $L^2$. 
\end{prop}

In the following, we also assume that $J$ attains its minimum in $BV\cap L^2$. This is for example true  if $G$ satisfies some coercivity hypothesis or if $G$ is non negative.\\

 Two fundamental applications of our method  are image denoising via total variation regularization and segmentation with geodesic active contours.\\

In the first problem, one starts with a corrupted image $f=\bar{u}+n$ and wants to find the clean image $\bar{u}$. Rudin, Osher and Fatemi proposed to look for an approximation of $\bar{u}$ by minimizing  
\[\intomega |Du|+\frac{\lambda}{2}\intomega (u-f)^2\]
This corresponds to $G(u)=\frac{\lambda}{2}\intomega(u-f)^2$ and $\partial \Omega_D =\emptyset$ in (\ref{generalprob}). For a comprehensive introduction to this subject, we refer to the lecture notes of Chambolle and al. \cite{antolinz}. Figure \ref{lenadenoise} shows the result of denoising using the algorithm of Chan and Zhu.\\

\begin{figure}[h]

 \centering
     \subfigure{
                
          \includegraphics[scale=0.3]{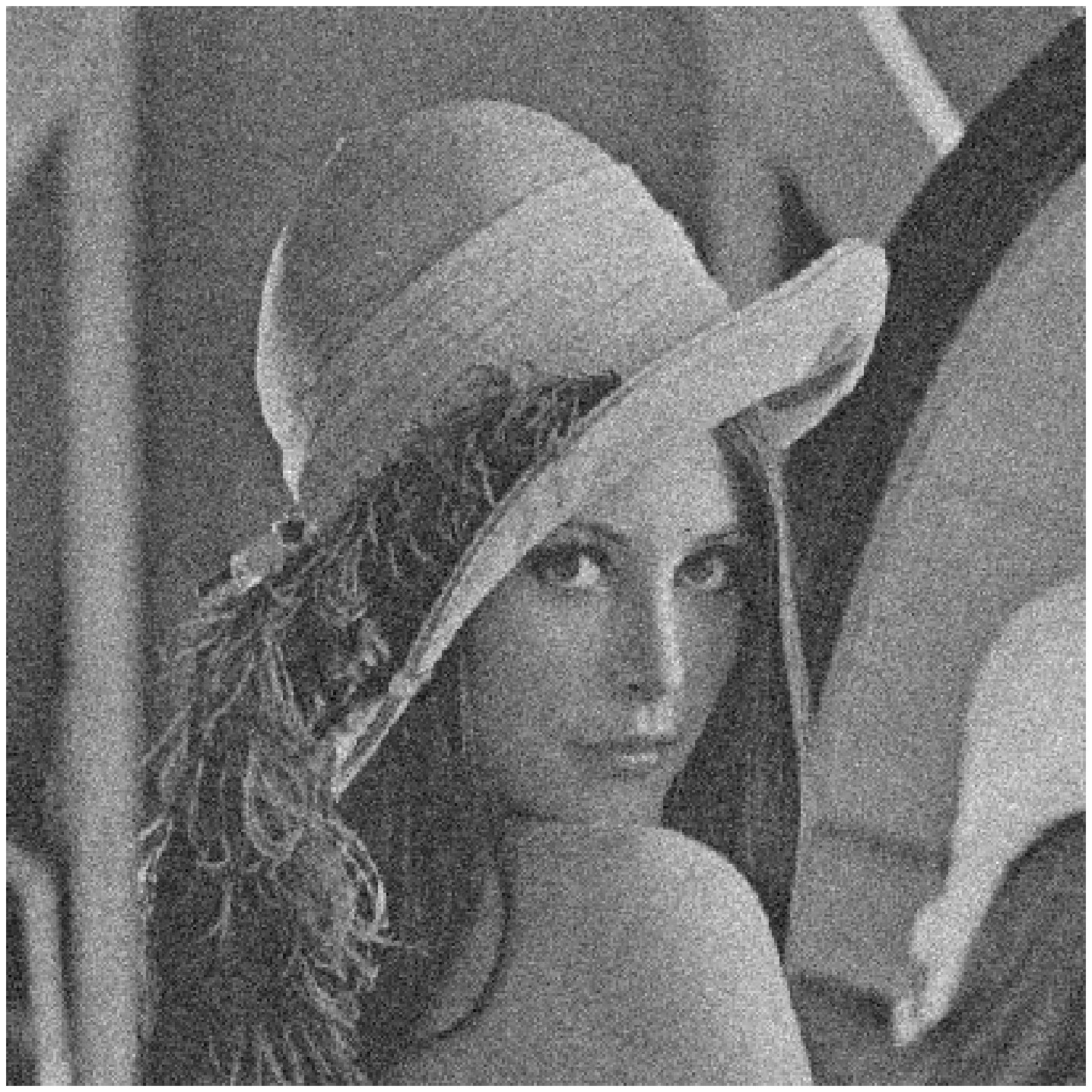}}
     \hspace{.3in}
     \subfigure{
         
          \includegraphics[scale=0.3]{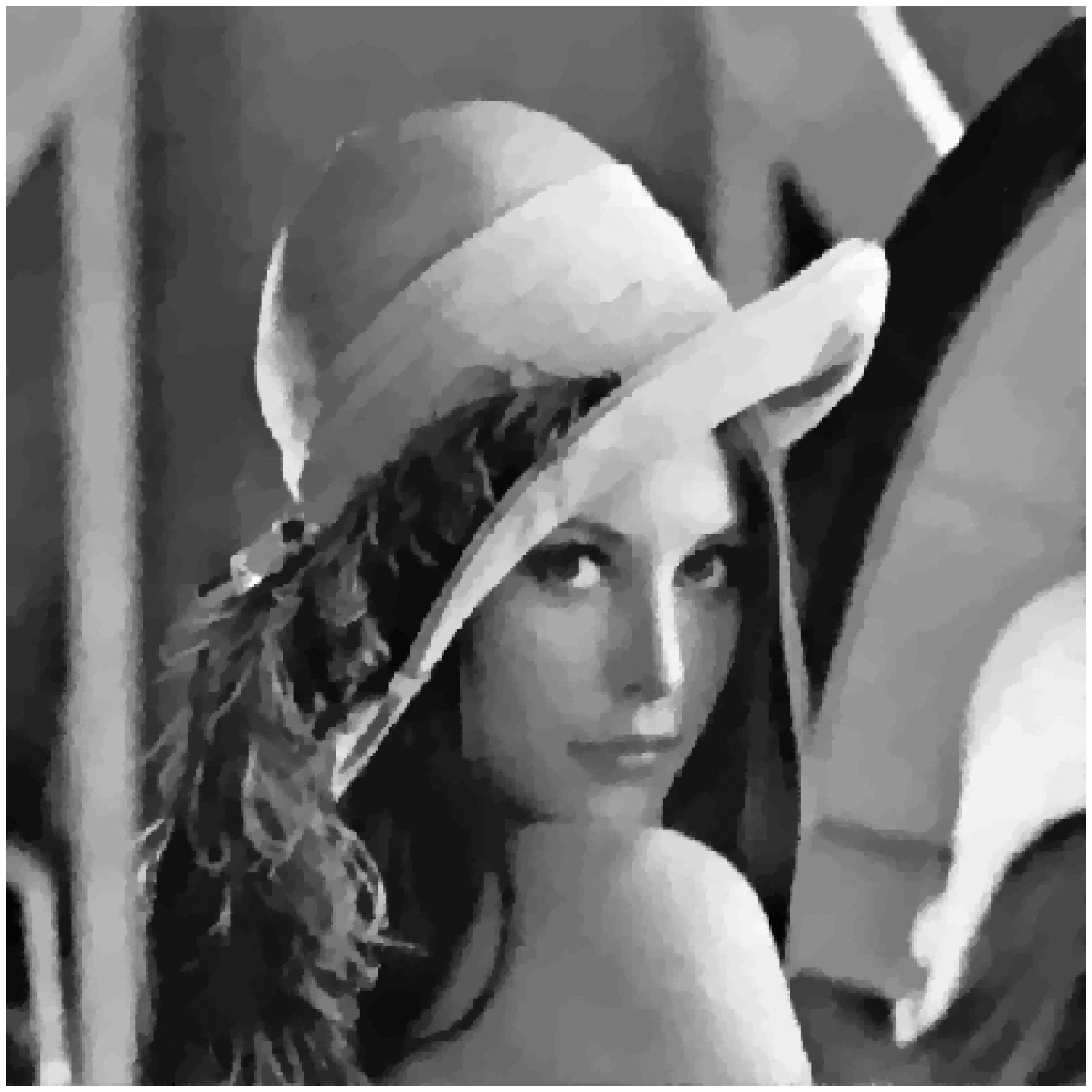}}%
\caption{Denoising using the ROF model}
\label{lenadenoise}
\end{figure}

 The issue in the second problem is to extract automatically the boundaries of an object within an image. We suppose that we are given two subsets $S$ and $T$ of $\partial \Omega$ such that $S$ lies inside the object that we want to segment and $T$ lies outside.  Caselles and al. proposed in \cite{geodesicactive} to associate a positive function $g$ to the image in a way that $g$ is  high where the gradient of the image is low and vice versa. The object is then segmented by minimizing
\begin{equation}\label{probg} \min_{E \supset S, \, E^c\supset T} \;  \int_{\partial E} g(s) ds \end{equation}

In order to simplify the notations, we will only deal with $g=1$ in the following. It is however straightforward to extend our discussion to general (continuous) $g$. The energy we want to minimize is thus $\intomega |D \chi_E|$. This functional is non convex but by the  coarea formula (see Ambrosio-Fusco-Pallara \cite{amb}), it can be  relaxed to functions $u\in[0,1]$. \\
Let $\varphi=1$ on $S$ and $\varphi=0$ on $T$. Letting $\partial \Omega_D=S\cup T$, and $f$ be a $L^2$ function, our problem can be seen as a special case of the prescribed mean curvature problem (in our original segmentation problem, $f=0$),
\begin{equation}\label{probrelax}\inf_{\stackrel{0\leq u\leq 1}{u=\varphi \textrm{ in } \partial \Omega_D}} \intomega |Du| +\intomega f u\end{equation}

If $u$ is a solution of (\ref{probrelax}), a minimizer $E$ of (\ref{probg}) is then given by any superlevel of $u$, namely $E=\{u>s\}$ for any $s \in ]0,1[$. This convexification argument is somewhat classical but more details can be found  in the lecture notes \cite{antolinz} Section 3.2.2.\\
It is however well known that in general the infimum is not attained  because of the lack of compactness for the boundary conditions in $BV$. Following the ideas of Giaquinta and al. \cite{giaquinta}  we have to relax the boundary conditions by adding a Dirichlet term $\intpartialD |u-\varphi|$ to the functional. We also have to deal with  the hard constraint, $0\leq u\leq 1$. This last issue will be discussed afterwards but it brings some mathematical difficulties that we were not able to solve. Fortunately, our problem is equivalent (see \cite{chambollecremerspock})  to the minimization of the unsconstrained problem
\[J(u)=\inf_{u\in BV(\Omega)} \intomega |Du|+\intpartialD |u-\varphi|+\intomega f^+ |u|+\intomega f^-|1-u| \]
Here $f^+=\max(f,0)$ and $f^-=\max(-f,0)$. \\
We give in Figure \ref{levures} the result of this segmentation on yeasts. The small square is the set $S$ and the set $T$ is taken to be the image boundary. The study of this problem was in fact our first motivation for this work.

\begin{figure}[h]
\centering

          \includegraphics[scale=0.35]{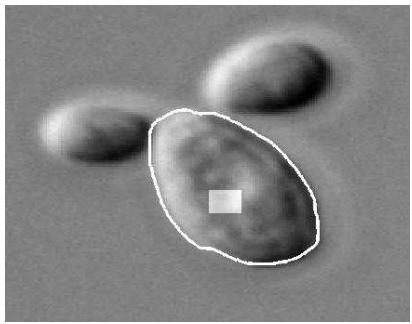}
    	
\caption{Yeast segmentation}
\label{levures}
\end{figure}

\subsection{Idea of the Primal-Dual method}
Formally, the idea behind the Primal-Dual method is using the definition of $\intomega |Du|$  (see Definition \ref{TV}) in order to write $J$ as

\[J(u)=\sup_{\stackrel{\xi \in \mathcal{C}^1_c(\Omega)}{|\xi|_{\infty} \leq 1}} K(u,\xi)\]

Where $K(u,\xi)=-\intomega u\dive(\xi) +\intpartialD |u-\varphi|+G(u)$. Then, finding a minimum of $J$ is equivalent to finding a saddle point of $K$. This is done  by a gradient descent in  $u$ and a gradient ascent in $\xi$. \\
Let $\ind(\xi)$ be the indicator function of the unit ball in $L^\infty$ (it takes the value 0 if $|\xi|_\infty \leq 1$ and $+\infty$ otherwise) and $\partial$ denotes the subdifferential (see Ekeland-Temam \cite{ekelandtemam} for the definition ). As 
\[K(u,\xi)=-\intomega u\dive(\xi) +\intpartialD |u-\varphi|+G(u)-\ind(\xi)\]
we have $\nabla_u K \simeq -\dive \xi +\partial G(u) $ and $\nabla_\xi K \simeq Du -\partial \ind(\xi)$. We are thus led to solve the system of PDEs:

\begin{equation}\begin{cases}\label{formal}
\displaystyle  \partial_t u=\dive (\xi)-\partial G (u)\\[8pt]
\displaystyle \partial_t \xi=Du-\partial \ind (\xi)\\[8pt]
+ \textrm{ boundary conditions}
\end{cases}
\end{equation}

 This system is almost the one proposed by Appleton and Talbot in \cite{AT} for the segmentation problem.\\

 Let us  remark that, at least formally, the  differential operator \\
$A(u,\xi)=\displaystyle \left(\begin{array}{c} -\, \text{div} \, \xi +\partial G(u)\\	
-Du +\partial \ind(\xi)\end{array}\right)$ verifies by Green's formula and the monotonicity of the subdifferential (see Proposition \ref{convmax}),
\[\langle A(u,\xi) , (u,\xi) \rangle =\langle \partial G(u),u\rangle +\langle\partial \ind (\xi),\xi\rangle\geq 0\] which means that $A$ is monotone (see Definition \ref{monotone}).\\
 In the next section we   recall some facts about the theory of maximal monotone operators and its applications for finding saddle points. In the last section we  use it to  give a  rigourous meaning to the hyperbolic system (\ref{formal}) together with  existence and uniqueness of solutions of the Cauchy problem.

 \section{Maximal Monotone Operators}
Following Br\'ezis  \cite{brezoperateur}, we present briefly in the first part of this section  the theory of maximal monotone operators. In the second part we show how this theory sheds light on the general Arrow-Hurwicz method. We mainly give results  found in Rockafellar's paper \cite{articlerockafellar}.

\subsection{Definitions and first properties of maximal monotone operators}

\begin{defin}
Let $X$ be an Hilbert space. An operator is a multivaluated mapping $A$ from $X$ into $\mathcal{P}(X)$. We call  $D(A)=\{x\in X \, / \,  A(x) \neq \varnothing\}$ the domain of $A$ and $\displaystyle R(A)=\bigcup_{x\in X} A(x)$ its range. We  identify $A$ and its graph in $X\times X$.
\end{defin}

\begin{defin}\label{monotone}
An operator $A$ is monotone if :
\[\forall x_1,x_2 \in D(A), \qquad \langle A(x_1)-A(x_2),x_1-x_2 \rangle \geq 0\]
or more precisely if for all $x_1^*\in A(x_1)$ and $x_2^* \in A(x_2)$,
\[\langle x_1^*-x_2^*,x_1-x_2 \rangle \geq 0\]

It is maximal monotone if it is maximal in the set of  monotone operators. The maximality is to be understood in the sense of graph inclusion.
\end{defin}

One of the essential results for us is the maximal monotonicity of the subgradient for convex functions.
\begin{prop}\emph{\cite{brezoperateur}}\label{convmax}
Let $\varphi$ be a proper lower-semi-continuous convex function on $X$ then $\partial \varphi$ is a maximal  monotone operator.
\end{prop}

Before stating the main theorem of this theory, namely the existence of solutions of the Cauchy problem $-u' \in A(u(t))$ we need one last definition.

\begin{defin}
Let $A$ be maximal monotone. For $x\in D(A)$ we call $A^\circ(x)$ the projection of $0$ on $A(x)$ (it exists since $A(x)$ is closed and convex, see Brézis \cite{brezoperateur} p. 20).
\end{defin}

We now turn to the theorem.
\begin{thm}\emph{\cite{brezoperateur}}\label{Cauchy}
Let $A$ be maximal monotone then for all  $u_0\in D(A)$, there exists a unique  function $u(t)$ from $[0,+\infty[$ into $X$ such that
\begin{itemize}
\item $u(t)\in D(A)$ for all $t>0$
\item $u(t)$ is Lipschitz continous on  $[0,+\infty[$, i.e $ u' \in L^\infty (0,+\infty;X)$ (in the sense of distributions) and 
\[ \left| u' \right|_{L^\infty (0,+\infty;X)} \leq |A^\circ(u_0)| \]
\item $\displaystyle - u'(t) \in A(u(t))$ for almost every $t$
\item $u(0)=u_0$
\end{itemize}
Moreover $u$ verifies,
\begin{itemize}
\item $u$ has a right derivative for every $t\in[0,+\infty[$ and    $\displaystyle -\frac{d^+ u}{dt} \in A^\circ(u(t))$
\item the function $t\rightarrow A^\circ(u(t))$ is right continuous and   $t\rightarrow |A^\circ(u(t))|$ is non increasing
\item if $u$ and $\hat{u}$ are two solutions then $|u(t)-\hat{u}(t)|\leq |u(0)-\hat{u}(0)|$
\end{itemize}
\end{thm}

\subsection{Application to Arrow-Hurwicz methods}

Let us now see how this theory can be applied for tracking saddle points. As mentioned before, we follow here \cite{articlerockafellar}. We start with some definitions.

\begin{defin} Let $X=Y\oplus Z$ where $Y$ and $Z$ are two Hilbert spaces. A proper saddle function on  $X$ is a function  $K$ such that :
\begin{itemize}
\item for all $y \in Y$, the function $K(y,\cdot)$ is convex 
\item for all $z\in Z$, the function $K(\cdot,z)$ is concave
\item there exists $x=(y,z)$ such that $K(y,z')<+\infty$ for all $z'\in Z$ and $K(y',z)>-\infty$ for all $y' \in Y$. The set of  $x$ for which it holds, is called the effective domain of $K$ and is noted $\dom   K$.
\end{itemize}
\end{defin}

\begin{defin}
A point $(y,z)\in X$ is called a saddle point of $K$ if 
\[K(y,z')\leq K(y,z) \leq K(y',z) \qquad \forall y' \in Y , \, \forall z' \in Z \]
\end{defin}

We then have, 

\begin{prop}
A point  $(y,z)$ is a saddle point of a saddle function $K$, if and only if 
\[ K(y,z)=\sup_{z'\in Z} \inf_{y' \in Y} K(y',z')= \inf_{y' \in Y} \sup_{z'\in Z} K(y',z') \]
\end{prop}
The proof of this proposition is easy and can be found in Rockafellar's book \cite{livrerockafellar} p.380.

The next theorem shows that the Arrow-Hurwicz method always provides a monotone operator.

\begin{thm}\emph{\cite{articlerockafellar}}
Let $K$ be a proper  saddle function. For $x=(y,z)$ let 
\[
T(x)=  \left\{  (y^*,z^*) \in Y^* \oplus Z^* / \begin{array}{c} y^* \text{ is a subgradient of } K(\cdot,z) \text{ in }  y    \\
   z^* \text{ is a subgradient of } -K(y,\cdot) \text{ in } z \end{array} \right\}
\]
Then $T$ is a monotone operator with $D(T) \subset \dom K$.
\end{thm}

We can now characterize the saddle points of $K$ using the operator $T$. 
\begin{prop}\emph{\cite{articlerockafellar}}
Let $K$  be a proper saddle function then a point  $x$  is a saddle point of $K$ if and only if $0\in T(x)$.
\end{prop}

\begin{remarque}
This property is to be compared with the minimality condition $0\in \partial f(x)$ for convex functions $f$.\end{remarque}

The next theorem shows that for regular enough saddle functions, the corresponding operator $T$ is maximal.

\begin{thm}\emph{\cite{articlerockafellar}}\label{maxpointselle}
Let  $K$ be a proper saddle function on $X$. Suppose that $K$ is lsc in $y$ and upper-semi-continuous in $z$ then $T$ is maximal monotone.
\end{thm}

\begin{proof}
 
We just sketch the proof because it will inspire us in the following. The idea is to use the equivalent theorem for convex functions. For this we ``invert'' the operator $T$ in the second variable. Let
\[H(y,z^*)=\sup_{z \in X}  \; \langle z^*, z \rangle + K(y,z) \]

The proof is then based on the following lemma :

\begin{lem}\label{lemmefond}
$H$ is a convex lsc function on $X$ and 
\[(y^*,z^*) \in T(y,z) \Leftrightarrow (y^*,z) \in \partial H(y,z^*) \]
\end{lem}
It is then not too hard to prove that $T$ is maximal.
 
\end{proof}

\section{Study of the Primal-Dual Method}
In this section, unless otherly stated, everything holds for general functionals $J$ of the type (\ref{generalprob}).\\

Before starting the study of the Primal-Dual method, let us  remind some facts about functions with bounded variation and pairings between measures and bounded functions. 

\begin{defin}\label{TV}
 Let $BV(\Omega)$ be the space of functions  $u$ in $L^1$ for which
\[\intomega |Du|:= \sup_{\stackrel{\xi \in \mathcal{C}^1_c(\Omega)}{|\xi|_\infty \leq 1}} \intomega u \dive \xi < +\infty \]
With  the norm $|u|_{BV}= \intomega|Du|+|u|_{L^1}$ it is a Banach space. We note the functional space $BV^2=BV(\Omega)\cap L^2$.
\end{defin}

\begin{prop}
 Let $u \in L^1(\Omega)$ then $u \in BV(\Omega)$ if and only if its distributional derivative $Du$ is a finite Radon measure. Moreover the total variation of $Du$ is equal to $\displaystyle \intomega |Du|$.
\end{prop}

More informations about functions with bounded variation, can be found in the books  \cite{amb} or  \cite{giusti}.\\
 Following Anzellotti \cite{anzel}, we  define $\int_\Omega [\xi,Du]$ which has to be understood as $\intomega  \xi \cdot Du$, for functions $u$ with bounded variation and bounded functions $\xi$ with divergence in $L^2$.

\begin{defin}
\begin{itemize}
	\item Let $X^2=\left\{ \xi \in (L^\infty(\Omega))^d \, / \; \dive \xi \in L^2(\Omega)\right\}$.
	\item  For $(u,\xi) \in BV^2 \times X^2$ we define the distribution $[\xi,Du]$ by 
 \[ \langle [\xi,Du], \varphi \rangle= -\int_\Omega u \varphi \dive (\xi) - \int_\Omega u \,  \xi \cdot \nabla \varphi  \qquad \forall \varphi \in \mathcal{C}^\infty_c(\Omega) \]
\end{itemize}

 \end{defin}
 
 \begin{thm}\emph{\cite{anzel}}
 The distribution $[\xi,Du]$ is a bounded Radon measure on $\Omega$ and if $\nu$ is the outward unit normal to $\Omega$, we have  Green's formula,
 \[ \int_\Omega [\xi,Du]=- \int_\Omega u \dive (\xi) +\int_{\partial \Omega} (\xi \cdot \nu) u\]
  \end{thm}

We now prove a useful technical lemma.
\begin{prop}\label{supxi}
 Let $u\in BV(\Omega)$ then 
\[\intomega |Du|=\sup_{\stackrel{\xi \in X^2}{|\xi|_\infty \leq 1}} \intomega [\xi,Du]\]
\end{prop}
\begin{proof}
By the definition of the total variation, 
\[\intomega |Du|\leq \sup_{\stackrel{\xi \in X^2}{|\xi|_\infty \leq 1}} \intomega [\xi,Du]\]
We thus only have to prove the opposite inequality.\\
Let $\mathcal{C}(\Omega)$ be the space of continuous functions on $\Omega$ then by Proposition 1.47 p.41 of the book \cite{amb},
 \begin{align*}
\intomega |Du|&= \sup_{\stackrel{\xi \in \mathcal{C}(\Omega)}{|\xi|_\infty \leq 1}} \intomega \xi\cdot Du\\
&\geq \sup_{\stackrel{\xi \in \mathcal{C}(\Omega)\cap X^2}{|\xi|_\infty \leq 1}} \intomega [\xi,Du]
\end{align*}
In the second inequality, the fact that $\displaystyle \intomega \xi\cdot Du=\intomega [\xi,Du]$ comes from Proposition 2.3 of \cite{anzel}. Let us also note that in the original Proposition 1.47 cited above, the supremum is taken over functions in $\mathcal{C}_c(\Omega)$ but a quick look to the proof shows that it can be enlarge to functions whose  support is not compact.\\

We now want to show that for every $\xi$ in $X^2$ with $|\xi|_\infty \leq 1$, there exists a sequence $ \xi_n$ in $X^2\cap \mathcal{C}(\Omega)$ with $|\xi_n|_\infty \leq 1$ such that $\displaystyle \intomega [\xi_n,Du]$ tends to $\displaystyle \intomega [\xi,Du]$, which would end the proof.\\

By Lemma 2.2 and Proposition 2.1 of \cite{anzel}, for every $\xi \in X^2$ with $|\xi|_\infty \leq 1$, we can find $ \xi_n \in X^2\cap \mathcal{C}(\Omega)$ with $|\xi_n|_\infty \leq 1$ and $[\xi_n,Du]$ tending to $[\xi,Du]$ in the sense of weak convergence of measures.\\
 The final step is now very similar to  the proof  of Theorem 4.2 of \cite{anzel}. \\

 Let $\varepsilon>0$ be given . There exists a number $\delta=\delta(\varepsilon)>0$ such that if we let $\Omega_\delta= \{ x \in \Omega \; |\;  dist (x,\partial \Omega) >\delta \}$
\[\int_{\Omega \backslash \Omega_\delta} |Du| \leq \varepsilon\]
Take $\eta$ a function of $\mathcal{C}_c(\Omega)$ with $\eta=1$ on $\Omega_\delta$  and $|\eta|_\infty\leq 1$, then
\begin{multline*}
\intomega [\xi_n,Du]-\intomega[\xi,Du]=\\
\left[ \intomega [\xi_n,Du]\eta-\intomega[\xi,Du]\eta \right]+\left[ \intomega [\xi_n,Du](1-\eta)-\intomega[\xi,Du](1-\eta) \right]
\end{multline*}
The first term in brackets goes to zero because of the weak convergence of $[\xi_n,Du]$ to $[\xi,Du]$. The second term can be  bounded by
\[2|\xi_n|_\infty \int_{\Omega \backslash \Omega_\delta} |Du|+2|\xi|_\infty \int_{\Omega\backslash \Omega_\delta} |Du| \leq 4\varepsilon\]
This shows the desired result. 
\end{proof}

 The next proposition gives a characterization of the minimizers of the functional $J$.

\begin{prop}\label{EL}
 Let $\displaystyle J(u)=\intomega |Du|+G(u)+\intpartialD |u-\varphi|$ then $u$ is a minimizer of $J$ in $BV^2$ if and only if there exists $\xi \in X^2$ such that
\[ \begin{cases}  \dive (\xi) \in \partial G(u) \\[8pt]
  \displaystyle \int_\Omega |Du| = \int_\Omega [\xi,Du]\\[8pt]
\xi \cdot \nu=0 \textrm{ in } \partial \Omega_N \, \textrm { and } \, (\xi \cdot \nu) \in sign(\varphi-u) \textrm { in } \partial \Omega_D

  \end{cases}\]

\end{prop}
We do not give the proof of this proposition here since it can be either found in Andreu and al. \cite{andreu} p.143 or derived more directly using the techniques we used in Proposition \ref{HROFmax} and Proposition \ref{HROFmaxbis}. \\

With these few propositions in mind we can turn back to the analysis of the Primal-Dual method. As noticed in the introduction, finding a minimizer of $J$ is equivalent to  finding a saddle point of

\[K(u,\xi) = \intomega [Du,\xi] +G(u)+\intpartialD |u-\varphi|-\ind(\xi)\]

The saddle function $K$ does not fulfill the assumptions of Theorem \ref{maxpointselle} since it is not lsc in $u$. However  staying in the spirit of Lemma \ref{lemmefond}, we set

\begin{align*}
 H(u,\xi^*)     &= \sup_{\stackrel{\xi \in X^2}{|\xi|_\infty \leq 1}}  \langle \xi,\xi^*\rangle +K(u,\xi)\\
		&= \sup_{\stackrel{\xi \in X^2}{|\xi|_\infty \leq 1}}  \langle \xi,\xi^*\rangle +\intomega [Du,\xi]  +G(u)+\intpartialD |u-\varphi|\\
		&= \intomega |Du+\xi^*|+G(u)+\intpartialD |u-\varphi|
\end{align*}

Where the last equality is obtained as in Proposition \ref{supxi}.
The function $H$ is then a convex lsc function on $L^2\times (L^2)^d$ hence $\partial H$ is maximal monotone. We are now able to define a maximal monotone operator $T$ by

\[T(u,\xi)= \left\{ (u^*,\xi^*) \, / \, (u^*, \xi) \in \partial H(u,\xi^*) \right\}\]

In order to compute $\partial H$, which gives the expression of $T$, we use the characterization of the subdifferential 
\[ (u^*,\xi) \in \partial H(u,\xi^*) \quad \Longleftrightarrow \quad \langle u^*,u \rangle + \langle \xi^*,\xi \rangle =H(u,\xi^*)+H^*(u^*,\xi) \]

A first step is thus to determine what $H^*$ is.

\begin{prop}\label{HROFmax}

We have 
\[D( H^* )= \left \{ (u^*,\xi) \, / \,  u^* \in L^2(\Omega) \text{ and } \xi \in X^2 \, , \, \xi \cdot \nu =0 \textrm{ in } \partial \Omega_N \, , \; |\xi|_\infty \leq 1 \right\}\] and

\[H^*(u^*,\xi)=G^*(u^*+\dive (\xi)) -\intpartialD (\xi \cdot \nu) \varphi. \]
\end{prop}

\begin{proof}
We start by computing the domain of $H^*$.\\
If  $(u^*,\xi) \in D(H^*)$ then there exists a constant $C$ such that for every \\$(u,\xi^*)\in BV^2 \times(L^2)^d$,

\[\langle u^*,u \rangle +\langle \xi^*,\xi \rangle -H(u,\xi^*) \leq C\]

Restraining to  $u \in H^1(\Omega)$ with $u\restD =0$ and $\xi^* \in (L^2)^d$, we find that

\[
 \langle u^*,u)+ \langle \xi^*,\xi \rangle-\intomega |\nabla u+\xi^*|-G(u) \leq C \]  from which 
\[ \langle \nabla u +\xi^*,\xi \rangle- \langle \nabla u,\xi\rangle +\langle u^*,u \rangle-\intomega |\nabla u+\xi^*|-G(u) \leq C
\]

Setting $\xi'=\nabla u +\xi^*$ and  taking the supremum over all $\xi' \in (L^2)^d$ we have that $|\xi|_\infty \leq 1$ and  for all $u \in H^1(\Omega)$ with $u\restD =0$ ,

\[- \langle \nabla u,\xi\rangle + \langle u^*,u \rangle \leq C +G(u)\]

Taking now $\tilde{u}=\lambda u$ with $\lambda$ positive and  reminding the form of $G$, it can be shown letting $\lambda$ tending to infinity, that for every $u \in H^1$ with $u\restD=0$,

\[- \langle \nabla u,\xi\rangle + \langle u^*,u\rangle \leq C |u|_{2} \]

This implies that $u^*+\dive \xi \in L^2$ hence $\dive \xi \in L^2$. Then by Green's formula in $H^1(\dive)$ (see Dautray-Lions \cite{dautraylions} p.205) we  have  $\xi \cdot \nu =0 \textrm{ in } \partial \Omega_N$.\\

Let us now compute $H^*$.\\
Let $(u^*,\xi)\in D(H^*)$,
\[H^*(u^*,\xi)=\sup_{\xi^* \in L^2} \sup_{u\in BV^2} \left\{ \langle  u^*,u\rangle +  \langle \xi^*,\xi\rangle -\int_\Omega |Du+\xi^*| -G(u)-\intpartialD |u-\varphi| \right\} \]

Let $\xi^* \in L^2$ be fixed. Then  by Lemma 5.2 p.316 of Anzellotti's paper \cite{anzel}, for every $u \in BV^2$ there exists $u_n \in \mathcal{C}^\infty \cap BV^2$ such that 

\begin{align*}
u_n & \stackrel{L^2}{\rightarrow} u \, , \, \quad (u_n) \restD =u\restD \qquad \text{ and } \\
 \int_\Omega |Du_n +\xi^*| & \rightarrow \int_\Omega |Du +\xi^*| 
\end{align*}

We can thus restrict the   supremum to functions  $u$ of class $\mathcal{C}^\infty(\Omega)$. We then have 

\begin{align*}
H^*(u^*,\xi) & =\sup_{u\in BV^2 \cap \mathcal{C}^\infty} \sup_{\xi\in L^2} \,\left\{  \langle u^*,u\rangle +  \langle \xi^*,\xi\rangle -\int_\Omega  |Du+\xi^*| -G(u) -\intpartialD |u-\varphi|\right\} \\
&=  \sup_{ u \in BV^2 \cap \mathcal{C}^\infty} \, \left\{ \langle u^*,u\rangle -  \langle \nabla u,\xi\rangle  -G(u) -\intpartialD |u-\varphi|\right\} \\
&=  \sup_{ u \in BV^2 } \, \left\{\langle u^*,u\rangle - \intomega [D u,\xi]  -G(u) -\intpartialD |u-\varphi|\right\} \\
&= \sup_{ u \in BV^2 } \, \left\{\langle u,u^*+\dive \xi\rangle  -G(u) -\intpartialD \left\{ |u-\varphi|+(\xi \cdot \nu) u \right\} \right\}
\end{align*}

Beware that $u\in BV^2 \cap \mathcal{C}^\infty$ implies that $\nabla u \in L^1$ and not  $\nabla u \in L^2$ but the density of $L^2$ in $L^1$ allows us to pass from the first equality to the second. The third equality follows from Lemma 1.8 of \cite{anzel}. We now have to show that we can take separately the supremum in the interior of  $\Omega$ and on the boundary $\partial \Omega_D$. \\

Let $f$ be in $L^1(\partial \Omega)$ and $v$ be in $ L^2(\Omega)$. We want to find $u_\varepsilon\in BV^2$ converging to $v$ in $L^2$ and such that  $(u_\varepsilon)\restD=f$.\\

By Lemma 5.5 of \cite{anzel} there is a  $w_\varepsilon \in W^{1,1}$ with $(w_\varepsilon)\restD =f$ and $|w_\varepsilon|_2\leq \varepsilon$. By density of $\mathcal{C}^\infty_c(\Omega)$ in $L^2$ we can find $v_\varepsilon \in \mathcal{C}^\infty_c(\Omega)$ with $|v_\varepsilon -v|_2 \leq \varepsilon$ We can then take $u_\varepsilon=v_\varepsilon+w_\varepsilon$. \\

This shows that 
\begin{align*}
H^*(u^*,\xi) & = \sup_{u \in L^2(\Omega)} \,\left\{ \langle u,u^*+\dive \xi \rangle  -G(u) \right\} - \inf_{u \in L^1} \intpartialD \left\{ |u-\varphi|+(\xi \cdot \nu) u \right\}\\
&= G^*(u^*+\dive (\xi)) -\intpartialD (\xi \cdot \nu) \varphi
\end{align*}
\end{proof}

We can now compute $T$

\begin{prop}\label{HROFmaxbis}
  Let $(u,\xi) \in BV^2\times X^2$ then, $(u^*,\xi^*) \in T(u,\xi)$ if and only if
  \[ \begin{cases} u^* + \dive (\xi) \in \partial G(u) \\[8pt]
  \displaystyle \int_\Omega |\xi^* +Du| =\langle \xi^*,\xi \rangle + \int_\Omega [\xi,Du]\\[8pt]
\xi \cdot \nu=0 \textrm{ in } \partial \Omega_N \, \textrm { and } \, (\xi \cdot \nu) \in sign(\varphi-u) \textrm { in } \partial \Omega_D

  \end{cases}\]
  \end{prop}
 
\begin{proof}
Let us first note that,
\begin{align}
\label{firstineq} G(u)+G^*(u^*+\dive (\xi)) &\geq  \langle u, u^*+\dive (\xi)\rangle\\
 \label{secondineq} \intomega |Du+\xi^*| &\geq \intomega [\xi,Du] +\intomega \xi^* \xi\\
\label{thirdineq} \displaystyle |u-\varphi|& \geq (\xi\cdot \nu)(\varphi-u)
\end{align}
where the second inequality is obtained arguing as in Proposition \ref{supxi}.\\
By definition, $(u^*,\xi^*)\in T(u,\xi)$ if and only if
\begin{align*}
 \langle u,u^*\rangle+ \langle \xi,\xi^*\rangle  = & \, H(u,\xi^*)+H^*(u^*,\xi) \qquad \\
   = &\intomega |Du+\xi^*|+G(u)+\intpartialD |u-\varphi| \\
  & +G^*(u^*+\dive (\xi)) - \intpartialD (\xi \cdot \nu) \varphi 
 \end{align*}
This shows that (\ref{firstineq}), (\ref{secondineq}) and (\ref{thirdineq}) must be equalities which is exactly   

\[ \begin{cases} u^* + \dive (\xi) \in \partial G(u) \\[8pt]
  \displaystyle \int_\Omega |\xi^* +Du| =\langle \xi^*,\xi\rangle + \int_\Omega [\xi,Du]\\[8pt]
 (\xi \cdot \nu) \in sign(\varphi-u) \textrm { in } \partial \Omega_D

 \end{cases}\]

Moreover, $\xi \cdot \nu=0$ in $\partial \Omega_N$ because $(u,\xi)\in D(T)$.
\end{proof}

\begin{remarque} \hspace{2cm} 
\begin{itemize}
		\item The condition $(\xi\cdot\nu) \in sign(\varphi-u)$ in $\partial \Omega_D$ is equivalent to 
\[\intpartialD |u-\varphi|+(\xi\cdot\nu)u=\inf_v \intpartialD |v-\varphi|+(\xi\cdot\nu)v\]
because inequality (\ref{thirdineq}) holds true for every $v$ and is an equality for $u$.
                 \item Whenever it has a meaning, it can be shown that the condition \[\displaystyle \int_\Omega |\xi^* +Du| =\langle \xi^*, \xi\rangle + \int_\Omega [\xi,Du]\] is equivalent to \[\xi^*+Du \in \partial\ind (\xi)\] so that we will not distinguish between these two notations.
		\item This analysis shows why the constraint  $u \in [0,1]$ is hard to  deal with. In fact, it imposes that $\dive(\xi)$ is a measure but not necessarily a  $L^2$ function. It is not easy to give a meaning to $\intomega Du \cdot \xi$ or to $(\xi \cdot \nu)$ on the boundary for such functions. However, when dealing with numerical implementations,  it is better to keep the constraint on $u$.
\end{itemize}

\end{remarque}

We can summarize those results in the following theorem which says that  the Primal-Dual Method is well-posed.

\begin{thm}
 For all $(u_0,\xi_0)\in \dom(T)$, there exists a unique $(u(t),\xi(t))$ such that
\begin{equation}\label{primaldual} \begin{cases} \partial_t u  \in \dive (\xi) -\partial G(u) \\[8pt]
 \partial_t \xi \in Du-\partial\ind (\xi)\\[8pt]
 (\xi \cdot \nu) \in sign(\varphi-u) \textrm { in } \partial \Omega_D \qquad \xi\cdot \nu=0 \textrm{ in } \partial \Omega_N\\[8pt]
(u(0),\xi(0))=(u_0,\xi_0)
\end{cases}\end{equation}

Moreover, the energy $|\dfrac{d^+u}{dt}|^2_2+|\dfrac{d^+\xi}{dt} |^2_2$ is non increasing and if $(\bar{u},\bar{\xi})$ is a saddle point of $K$, ${|u-\bar{u}|^2_2+|\xi-\bar{\xi}|^2_2}$ is also non increasing.
\end{thm}
\begin{proof}
 The operator $T$ is maximal monotone hence Theorem \ref{Cauchy} applies and gives the result.
\end{proof}

\begin{remarque}
 This theorem also shows that whenever $J$ has a minimizer, $K$ has  saddle points. This is  because stationnary points of the system (\ref{primaldual}) are minimizers of $J$ (verifying the Euler-Lagrange equation for $J$, remind Proposition \ref{EL}).
\end{remarque}

For the Rudin-Osher-Fatemi model, one can show that there is convergence of $u$ to the minimizer of the functional $J$ and obtain \textit{a posteriori} estimates. 
\begin{prop}
Let $\displaystyle G=\frac{\lambda}{2}\intomega (u-f)^2$ and $\partial \Omega_D =\emptyset$. Then if $\bar{u}$ is the minimizer of $J$, every solution of (\ref{primaldual}) converges in $L^2$ to $\bar{u}$. Furthermore, 

\[|u-\bar{u}|_2 \leq \frac{1}{2}\left(\frac{1}{\lambda}|\partial_t u|_2+\sqrt{\frac{|\partial_t u|^2_2}{\lambda^2}+\frac{8|\Omega|^{\frac{1}{2}}}{\lambda}|\partial_t \xi|_2}\right) \]

\end{prop}

\begin{proof}
  Let $(\bar{u},\bar{\xi})$ be such that $0\in T(\bar{u}, \bar{\xi})$. Let $e(t)=|u(t)-\bar{u}|_2^2$ and \\ $ g(t)=|\xi(t)-\bar{\xi}|_2^2$. We  show that 
 \begin{equation}\label{gronwall} 
\frac{1}{2}  (e+g)' \leq - \lambda e \end{equation}
Indeed, by definition of the flow,
\begin{align*}
 \intomega [\xi,Du]-\langle \xi , \partial_t \xi \rangle & \geq \intomega [\bar{\xi},Du]-\langle \bar{\xi}, \partial_t \xi\rangle \qquad \textrm{ and } \\
 \intomega [\bar{\xi},D\bar{u}]-\langle \bar{\xi} , \partial_t \bar{\xi}\rangle & \geq \intomega [{\xi},D\bar{u}]-\langle {\xi}, \partial_t \bar{\xi}\rangle 
\end{align*}
Summing these two we find,
\[
\intomega [\xi-\bar{\xi},D(u-\bar{u})] \geq \langle \xi-\bar{\xi}, \partial_t \xi -\partial_t \bar{\xi} \rangle  
\]
We thus have
\begin{align*}
 \frac{1}{2}  (e+g)'&=\langle u-\bar{u}, \partial_t u -\partial_t \bar{u} \rangle +\langle \xi-\bar{\xi}, \partial_t \xi -\partial_t \bar{\xi} \rangle\\
&\leq \langle u-\bar{u}, \dive (\xi-\bar{\xi})-\lambda(u-\bar{u})\rangle+\intomega [\xi-\bar{\xi},D(u-\bar{u})]\\
&=-\lambda e
\end{align*}

 The functions $e$ and $g$ are Lipschitz continuous. Let $L$ be the  Lipschitz constant of $e$ and let  $h=e+g$. \\
 
Let us show by contradiction that $e$ tends to zero when $t$ tends to infinity.\\

 Suppose that there exists $\alpha>0$ and $T>0$ such that $e\geq \alpha$ for all $t>T$, then we would have $h' \leq -\lambda \alpha$ and   $h$ would tend to minus infinity which is impossible by positivity of $h$. Hence 
 \[ \forall \alpha>0 \; \forall T>0 \; \exists t\geq T \qquad e(t) \leq \alpha \]
 
 Suppose now the existence  of  $\varepsilon >0$ such that for all $T\geq 0$ there exists $t\geq T$ with $e(t) \geq \varepsilon$. \\
 By continuity of $e$, there exists a sequence $(t_n)_{n\in \mathbb{N}}$ with $\displaystyle \lim_{n \rightarrow +\infty} t_n=+\infty $ such that 
 \[ e(t_{2n})=\frac{\varepsilon}{2} \qquad e(t_{2n+1})=\varepsilon \]
 
 Moreover, on  $[t_{2n-1}, t_{2n}]$, we have  $e(t) \geq \frac{\varepsilon}{2}$. We then find that 
 \begin{align*}
 |e(t_{2n})-e(t_{2n-1})| & \leq L(t_{2n} -t_{2n-1}) \quad \text{ so} \\
 \frac{\varepsilon}{2L} & \leq t_{2n}-t_{2n-1}
 \end{align*}
 
 From which we see that,
 \begin{align*}
 h(t_{2n+2})& = h(t_{2n+1})+\int_{t_{2n+1}}^{t_{2n+2}} h'(t) \, dt\\
 & \leq h(t_{2n+1})-\varepsilon \lambda (t_{2n+2}-t_{2n+1})\\
 & \leq h(t_{2n})-\frac{\lambda \varepsilon^2}{2L}
 \end{align*}
 
 This shows that $\displaystyle \lim_{t\rightarrow + \infty} e(t)=0$.\\

We now prove the \textit{a posteriori} error estimate.\\

 We have that
\begin{align*}
 u&=f+\frac{1}{\lambda}(\dive \xi-\partial_t u)\\
\bar{u}&=f+\frac{1}{\lambda}\dive \bar{\xi}
\end{align*}
 Which leads to
\begin{align*}
 |u-\bar{u}|^2_2&=\frac{1}{\lambda}\langle \dive(\xi-\bar{\xi})-\partial_t u,u-\bar{u}\rangle\\
 		&=\frac{1}{\lambda}\left[\langle\dive(\xi-\bar{\xi}),u-\bar{u}\rangle-\langle \partial_t u,u-\bar{u}\rangle\right]\\
 		&=\frac{1}{\lambda}\left[-\langle \xi-\bar{\xi},Du-D\bar{u}\rangle-\langle \partial_t u,u-\bar{u}\rangle\right]\\
 		&\leq \frac{1}{\lambda}\left[\intomega |Du|-\intomega [\xi,Du]+|\partial_t u|_2|u-\bar{u}|_2 \right]
\end{align*}

 Where the last inequality follows from $\displaystyle \intomega [\bar{\xi},Du]\leq \intomega |Du|$ and \\$\displaystyle \intomega \bar{\xi} \cdot D\bar{u}=\intomega |D\bar{u}| \geq0$.\\
Studying the inequality $X^2 \leq A+B X$, we can deduce that  
\[|u-\bar{u}|_2\leq \frac{1}{2}\left(\frac{1}{\lambda}|\partial_t u|_2+\sqrt{\frac{|\partial_t u|^2_2}{\lambda^2}+\frac{4}{\lambda}(\intomega|Du|-\intomega [\xi,Du])}\right)\]

The estimate follows from the fact that
\begin{align*}
\intomega |-\partial_t \xi+Du|		&=\intomega [\xi,Du]-\intomega \partial_t \xi \cdot \xi \quad \textrm{ thus }\\
\intomega|Du|-\intomega |\partial_t \xi| &\leq  \intomega [\xi,Du]-\intomega \partial_t \xi \cdot \xi \quad \textrm{ hence }\\
\intomega|Du|-\intomega [\xi,Du]&\leq 2\intomega|\partial_t \xi| \leq 2|\Omega|^{\frac{1}{2}}|\partial_t \xi|_2
\end{align*}

\end{proof}

Following the same lines, we can show \textit{a posteriori} error estimates for general finite difference scheme. Indeed if  $\gradh$ is any discretization of the gradient and if $\divh$ is defined as $-(\gradh)^*$, the associated algorithm is

\begin{equation}\label{PD}\begin{cases}
\xi^n=P_{B(0,1)}(\xi^{n-1}+\delta \tau^n \gradh u^{n-1})\\[8pt]
   u^n=u^{n-1}+\delta t^n (\divh \xi^{n}-\lambda(u^{n-1}-f))
  \end{cases}\end{equation}
Where $P_{B(0,1)}(\xi)_{i,j}=\dfrac{\xi_{i,j}}{\max(|\xi_{i,j}|,1)}$ is the componentwise projection of $\xi$ on the unit ball. This algorithm is exactly the one proposed by Chan and Zhu in \cite{zhuchan}. We can associate to this system  a discrete energy, 
\[J_h(u)= \sum_{i,j} |\gradh u|_{i,j} +\frac{\lambda}{2}\sum_{i,j}|u_{i,j}-f_{i,j}|^2\]

The algorithm (\ref{PD}) could have been directly derived from this discrete energy using the method of Chan and Zhu \cite{zhuchan} (which is just the discrete counterpart of our continuous method). Hence, the next proposition gives a stopping criterion  for their algorithm.

\begin{prop}\label{posterioriROF}
Let $N\times M$ be the size of the discretization grid and  $\bar{u}$ be the minimizer of $J_h$ then
 \[|u^n-\bar{u}|_2 \leq \frac{1}{2}\left(\frac{1}{\lambda}|\partial_t u^n|_2+\sqrt{\frac{|\partial_t u^n|^2_2}{\lambda^2}+\frac{8\sqrt{N\times M}}{\lambda}|\xi^{n}_t|_2}\right) \]
Where $\displaystyle \partial_t u^n=\frac{u^{n+1}-u^n}{\delta t^{n+1}}$ and $\displaystyle \partial_t \xi^{n}=\frac{\xi^{n+1}-\xi^{n}}{\delta \tau^{n+1}}$.
\end{prop}

The proof of this discrete estimate is almost the same as for the continuous one. We give it in the appendix.\\
\begin{remarque}
 In opposition to the continuous framework where we were able to prove a convergence result,  no fully satisfactory statement is known in the discrete framework. For some partial results we refer to  Esser and al. \cite{esser}  and to Chambolle and Pock \cite{chambpock}. These works mainly focus on slight modifications of the Primal-Dual algorithm (\ref{PD}) but they also show that in some restricted cases the algorithm of Chan and Zhu converges.
\end{remarque}

For the general problem, there is no uniqueness for the minimizer (for example in the segmentation problem) and hence convergence may not occur or be hard to prove. Indeed, even when uniqueness holds, we can have non vanishing oscillations. For example in the simpler  one dimensional problem  
\[ \min_{u\in BV([0, 1])} \int_0^1 |u'|\]
the unique minimizer is $u=0$ but $u(t,x)=\frac{1}{2} \cos(\pi x)\sin(\pi t)$ and \\$\xi(t,x)=\frac{1}{2}\sin(\pi x) \cos( \pi t)$ gives a solution to the associated  PDE system which does not converge to a saddle point. In this example, the energy is constant hence not converging to zero. We can however show general \textit{a posteriori} estimates for the energy.

\begin{prop}\label{estimateAT}
 For every saddle point $(\bar{u},\bar{\xi})$  and every $(u_0,\xi_0)$, the solution $(u(t),\xi(t))$ of (\ref{primaldual}) satisfies

\[|J(u)-J(\bar{u})|\leq \left(\sqrt{|u_0-\bar{u}|^2_2+|\xi_0-\bar{\xi}|^2_2} \right)|\partial_t u|_2+2|\Omega|^{\frac{1}{2}}|\partial_t \xi|_2\] 
\end{prop}

\begin{proof}
 
Let $(\bar{u},\bar{\xi})$ be a saddle point and  $(u(t),\xi(t))$ be a solution of (\ref{primaldual}).
 \[J(u)-J(\bar{u})=\intomega |Du|+\intpartialD |u-\varphi|-\intomega |D\bar{u}|-\intpartialD |\bar{u}-\varphi|+G(u)-G(\bar{u})\]

By definition of the operator $T$ we have
\begin{align*}
\intomega [\xi,Du]-\intomega \partial_t \xi \cdot \xi=& \intomega |Du-\partial_t \xi|\\
 		   \geq & \intomega |Du|-\intomega|\partial_t \xi| 
\end{align*}

This shows that 
\begin{equation}\label{ineq1}\intomega |Du| \leq \intomega [\xi,Du]+2\intomega |\partial_t \xi| \end{equation}
On the other hand,
 \[ \intomega [\xi,Du]+\intpartialD |u-\varphi|=-\intomega u \dive \xi +\intpartialD \left\{(\xi\cdot\nu)u+|u-\varphi|\right\} \]
Applying $\displaystyle \intpartialD \left\{(\xi\cdot\nu)u+|u-\varphi|\right\}=\inf_v \intpartialD \left\{(\xi\cdot\nu)v+|v-\varphi|\right\}$ (remember the Remarks after Proposition \ref{HROFmaxbis}) to $v=\bar{u}$ we have
\begin{align*}
 \intomega [\xi,Du]+\intpartialD |u-\varphi|-\intpartialD |\bar{u}-\varphi|&\leq -\intomega u\dive \xi + \intpartialD (\xi \cdot \nu)\bar{u}\\
		&= -\intomega u\dive \xi +\intomega \bar{u}\dive \xi+\intomega [\xi, D\bar{u}]\\
		&= \intomega (\bar{u}-u) \dive \xi +\intomega [\xi, D\bar{u}]
\end{align*}
This and (\ref{ineq1}) show that
\[J(u)-J(\bar{u}) \leq \intomega (\bar{u}-u) \dive \xi +\intomega [\xi, D\bar{u}]+2\intomega |\partial_t \xi|-\intomega |D\bar{u}|+G(u)-G(\bar{u})\]
If we now use the definition of the subgradient to get
\[ G(u)-G(\bar{u})\leq \langle \dive (\xi) -\partial_t u,u-\bar{u}\rangle \]
 we find with Cauchy-Schwarz's inequality,

\begin{align*}
 J(u)-J(\bar{u})&\leq 2|\Omega|^{\frac{1}{2}}|\partial_t \xi|_2+\intomega (\bar{u}-u)\partial_t u+\intomega [\xi, D\bar{u}]-\intomega |D\bar{u}|\\
		&\leq 2|\Omega|^{\frac{1}{2}}|\partial_t \xi|_2+|\bar{u}-u|_2 |\partial_t u|_2
\end{align*}

Which gives the estimate reminding that $\sqrt{|u-\bar{u}|^2_2+|\xi-\bar{\xi}|^2_2}$ is  non increasing.
\end{proof}
\begin{remarque} \hspace{2cm}\\
 Supported by numerical evidence, we can conjecture that whenever the constraint on $\xi$ is saturated somewhere,  convergence of $u$ occurs. It might however be also necessary to add the  constraint $u \in [0,1]$ in order to have this convergence.
\end{remarque}

Considering a finite difference scheme, just as for the Rudin-Osher-Fatemi model, we can define a discrete energy $J_h$ and show the corresponding \textit{a posteriori} estimate.

\begin{prop}
 If $\bar{u}$ is a minimizer of $J_h$ and $(u^n,\xi^n)$ is defined by 

\[\begin{cases}
\xi^n=P_{B(0,1)}(\xi^{n-1}+\delta \tau^n \gradh u^{n-1})\\[8pt]
   u^n=u^{n-1}+\delta t^n( \divh \xi^{n}-p^n)

  \end{cases}\]
with $p^n \in \partial G^h(u^{n-1})$ then
\[|J_h(u^n)-J_h(\bar{u})|\leq 2\sqrt{N\times M}|\partial_t \xi^{n}|+|\partial_t u^{n}||u^{n-1}-\bar{u}|\]

\end{prop}
We omit the proof because it is exactly the same as for Proposition \ref{estimateAT}. 
\begin{remarque} \hspace{2cm}
\begin{itemize}
	\item The boundary conditions are hidden here in the operator $\gradh$.
	\item In the discrete framework, the estimate involves $|u^n-\bar{u}|$ which  can not be easily bounded by the initial error.
\end{itemize}

\end{remarque}

\section{Numerical Experiments}
To illustrate the relevance of our \textit{a posteriori} estimates, we first consider the simple example of denoising  a rectangle (see Figure \ref{denoisecarre}). We then compare the  \textit{a posteriori} error bound with the "true" error. We use the relative $L^2$ error defined as $\dfrac{|u^n-\bar{u}|}{|\bar{u}|}$ and ran the algorithm of Chan and Zhu with $\lambda=0.005$ and fixed time steps verifying $\lambda \delta t=1$ and $\delta \tau=\frac{\lambda}{5}$. With this choice of parameters convergence is guaranteed by the work of Esser and al. \cite{esser}.  The minimizer $\bar{u}$ is computed by the algorithm after 50000 iterations. Figure \ref{courbe} shows that the \textit{a posteriori} bound is quite sharp. \\

\begin{figure}[ht]

 \centering
     \subfigure{
                
          \includegraphics[width=5cm]{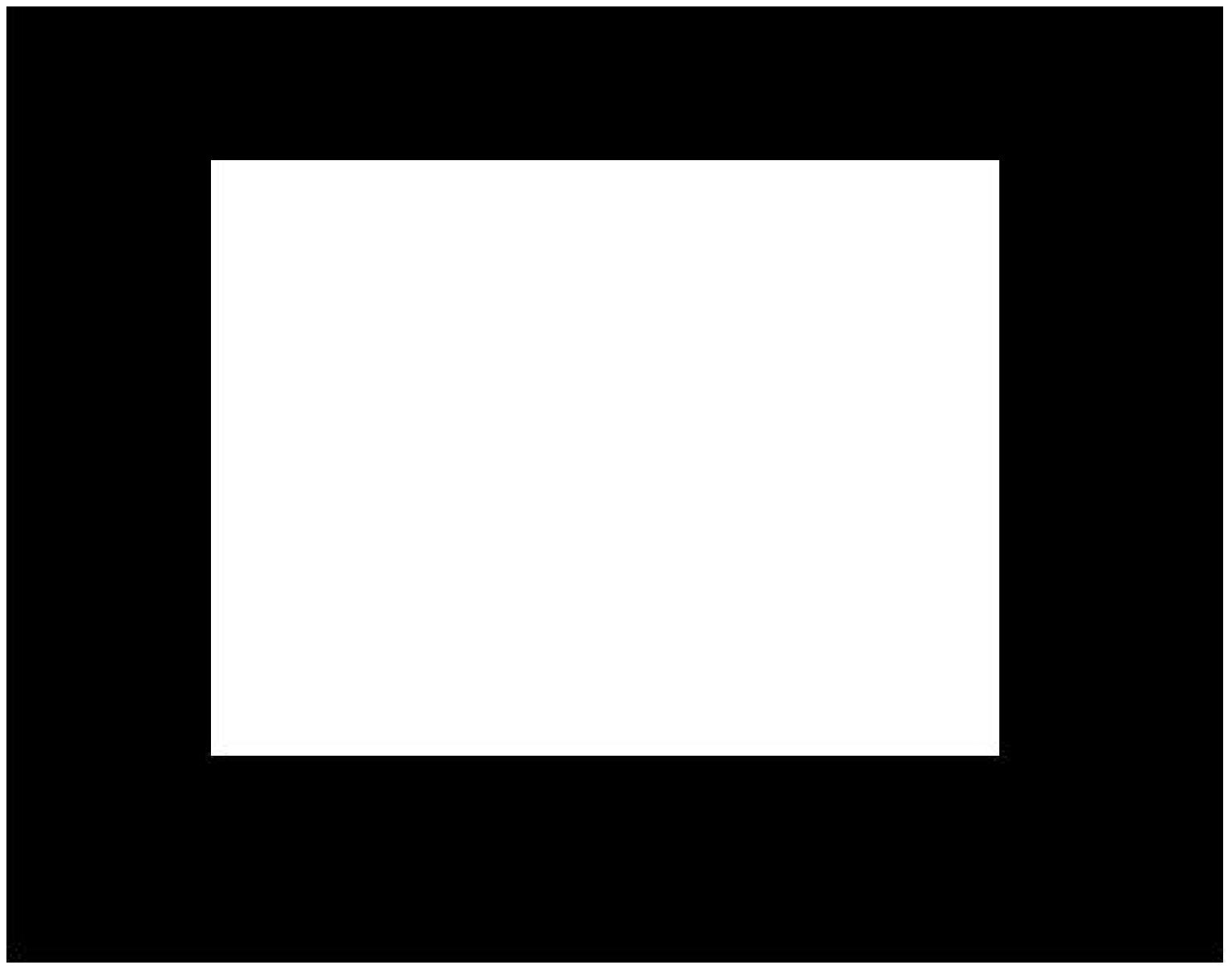}}
     \hspace{.3in}
     \subfigure{
         
          \includegraphics[width=5cm]{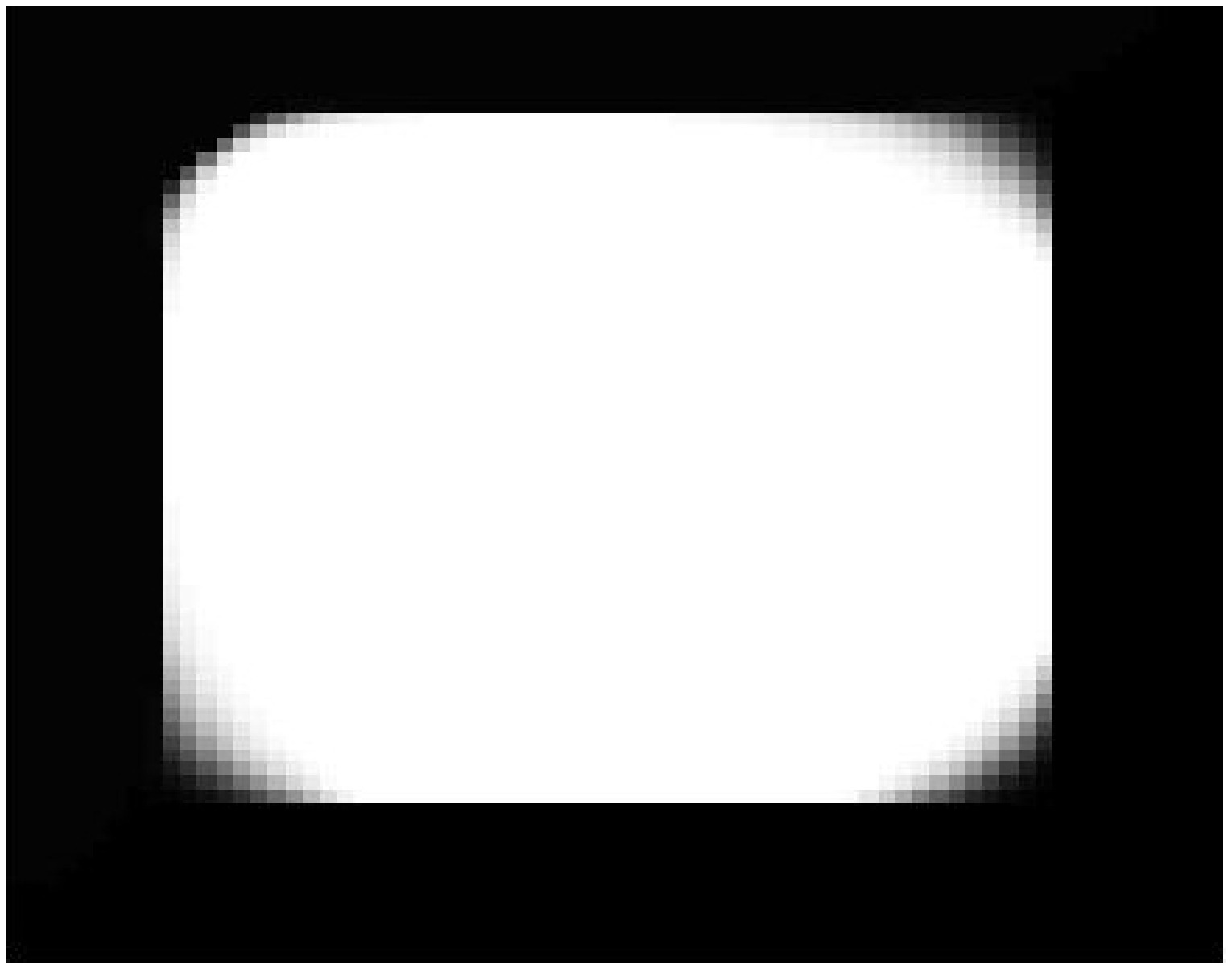}}%
\caption{Denoising of a rectangle using the ROF model}
\label{denoisecarre}
\end{figure}

\begin{figure}[ht]

 \centering
\includegraphics[scale=0.4]{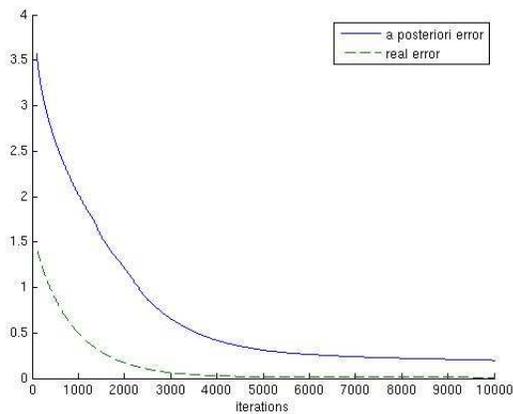}
\caption{Comparison of the relative $L^2$ error with the predicted \textit{a posteriori} bound.}
\label{courbe}
\end{figure}

The second experiment is performed on the yeast segmentation of Figure \ref{levures}. The solution was computed with the algorithm of Chan and Zhu using as weight function $g$ the one proposed by Appleton and Talbot \cite{AT}. We used this time the error  $|J_h(u^n)-J_h(\bar{u})|$ and ran the algorithm with  $\delta t=0.2$ and $\delta \tau =0.2$. For this problem there is no proof of convergence of the algorithm. The minimizer $\bar{u}$  is computed by the algorithm after 50000 iterations. We can see on Figure \ref{courbeyeasts} that for this problem, the \textit{a posteriori} estimate is not so sharp. We must also notice that in general we do not know $\bar{u}$. \\

\begin{figure}[ht]

 \centering
\includegraphics[scale=0.4]{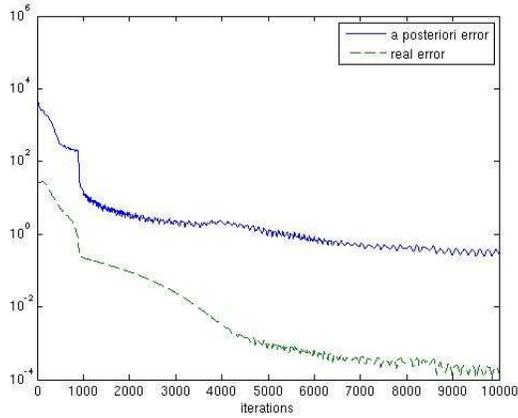}
\caption{Comparison for the segmentation problem.}
\label{courbeyeasts}
\end{figure}

In the last numerical example, we compare the results obtained by the algorithm of Appleton and Talbot (see \cite{AT}) with those obtained by a classical discretization of the total variation. In Figure \ref{cercle}, we can see the denoising of a disk with these two methods for $\lambda=0.003$.  We used the algorithm of Chan and Zhu \cite{zhuchan} to compute the minimization of the discrete total variation.\\
 Looking at the top right corner (see Figure \ref{corner}), we can see that the result is more accurate and less anisotropical for the algorithm of Appleton and Talbot than for the scheme of Chan and Zhu. These results are to be compared with those obtained by Chambolle and al. for the so-called ``upwind'' discrete $BV$ norm in \cite{chambisot}.

\begin{figure}[ht]

 \centering
     \subfigure{
                
          \includegraphics[width=6.2cm]{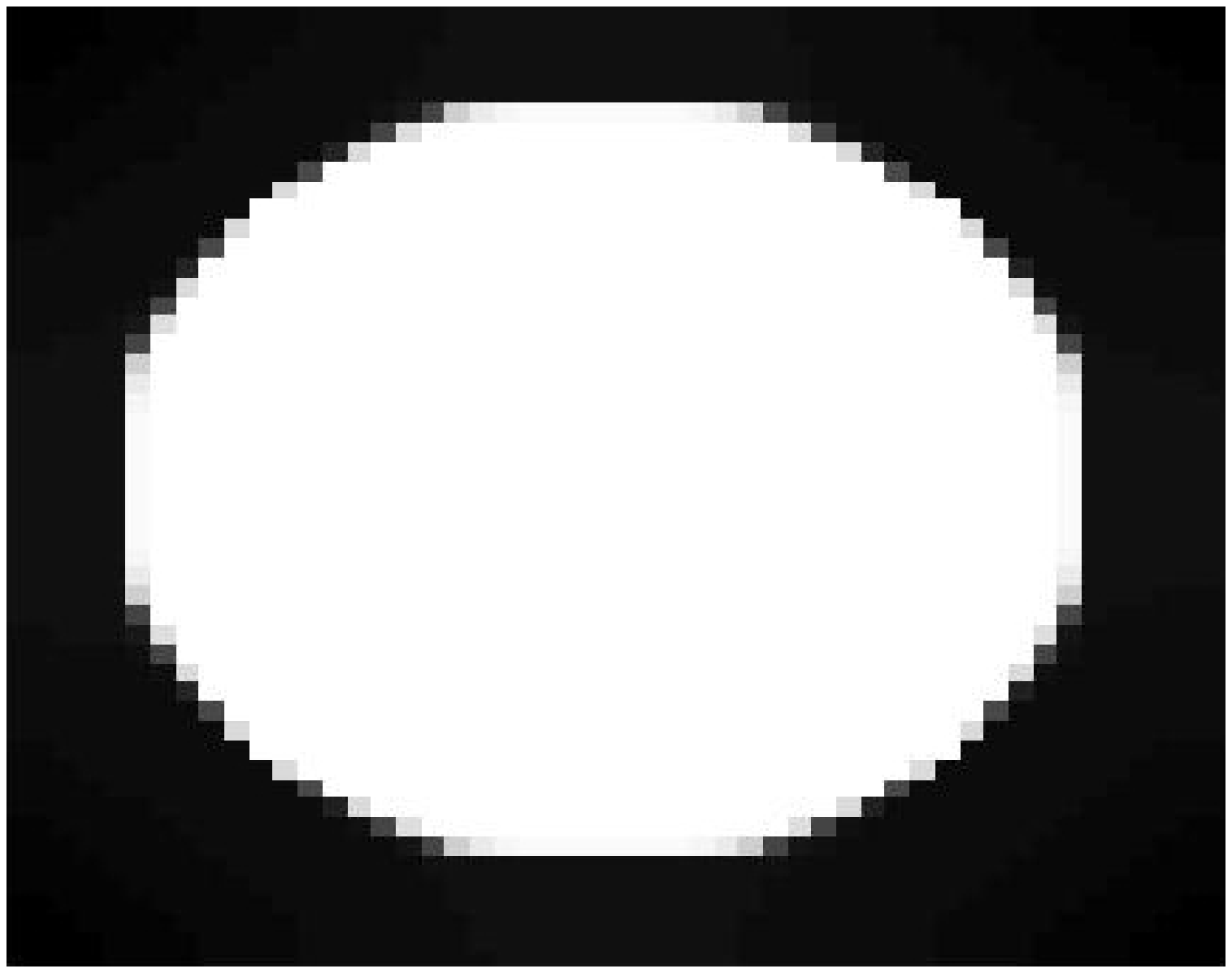}}
     \hspace{.3in}
     \subfigure{
         
          \includegraphics[width=6.2cm]{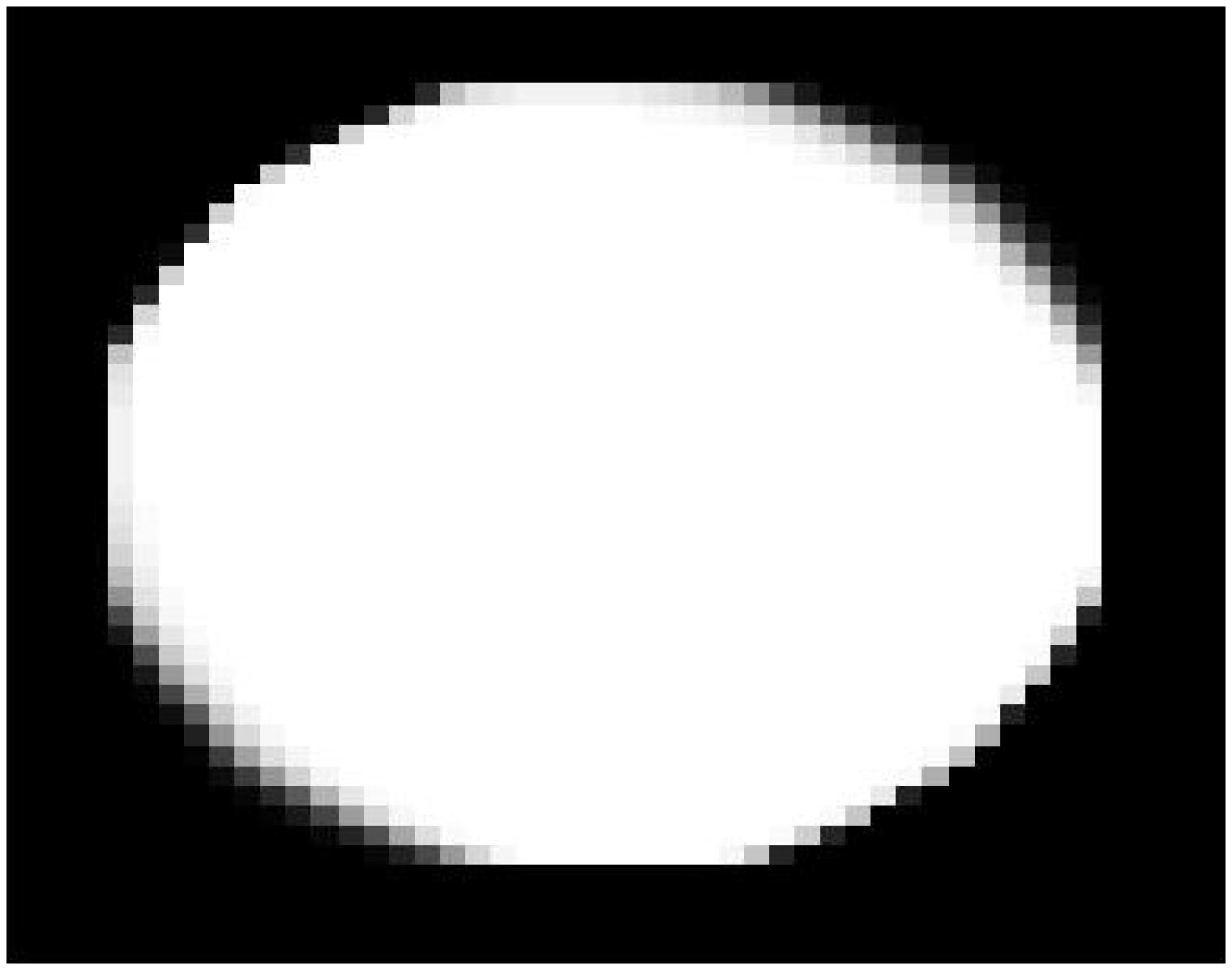}}%
\caption{Denoising of a disk using the algorithm of Appleton-Talbot (left) and Chan-Zhu (right)}
\label{cercle}
\end{figure}

\begin{figure}[ht]

 \centering
     \subfigure{
                
          \includegraphics[width=5cm]{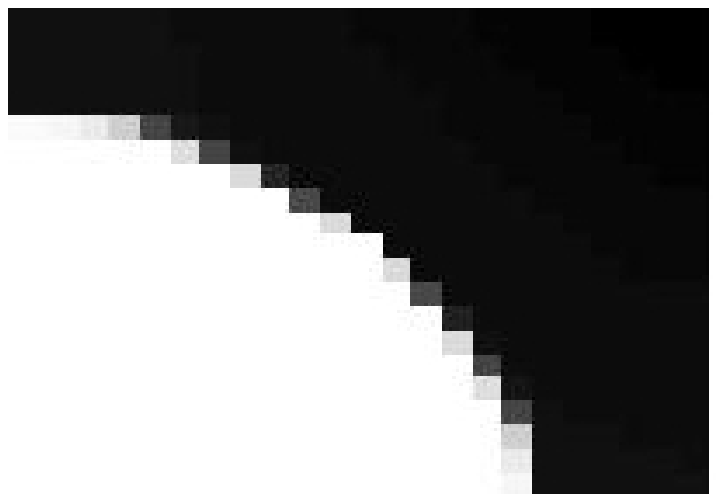}}
     \hspace{.3in}
     \subfigure{

          \includegraphics[width=5cm]{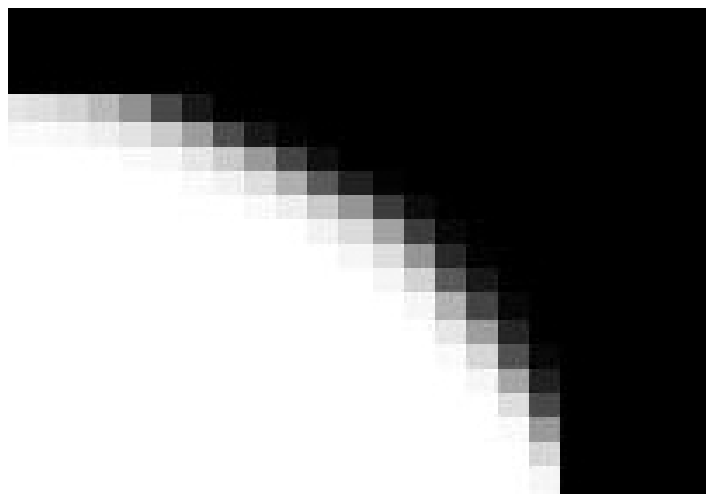}}%
\caption{Top right corner of the denoised disk, Appleton-Talbot (left) and Chan-Zhu (right)}
\label{corner}
\end{figure}

\section{Conclusion}
In this article we have shown the well posedness of the continuous Primal-Dual method proposed by Appleton and Talbot for solving problems arising in imaging. We have also proved for the ROF model, that in the continuous setting there is convergence towards the minimizer. We then derived some \textit{a posteriori} estimates. Numerical experiments have illustrated that if these estimates are quiet sharp for the ROF model, they should be improved for applications to other problems. \\
The continuous framework leaves the way open to a wide variety of numerical schemes, ranging from finite differences to finite volumes. Indeed, by designing algorithms solving the system of PDEs (\ref{primaldual}) one can expect to find accurate algorithms for computing solutions of variational problems involving a total variation term. 
\appendix
\section{Proof of Proposition \ref{posterioriROF}}

For notational convenience, we present the proof for $\lambda=1$. Let $\bar{u}$ be the minimizer of $J_h$ then there exists $\bar{\xi}$  such that $|\bar{\xi}|_\infty \leq 1$ and
\[\begin{cases}
   \sum_{i,j} |\gradh \bar{u}|_{i,j}=\langle \gradh \bar{u},\bar{\xi}\rangle\\[8pt]
\bar{u}=\divh \bar{\xi}+f
  \end{cases}\]

Reminding that $u^n=f+\divh \xi^{n+1}-\partial_t u^n$ we get
\begin{align*}
 |u^n-\bar{u}|^2 &=\langle\divh(\xi^{n+1}-\bar{\xi})-\partial_t u^n,u^n-\bar{u}\rangle\\
 		&=-\langle\xi^{n+1}-\bar{\xi},\gradh u^n-\gradh \bar{u}\rangle-\langle \partial_t u^n,u^n-\bar{u}\rangle\\
 		&\leq \langle\bar{\xi}-\xi^{n+1},\gradh u^n\rangle+|\partial_t u^n| |u^n-\bar{u}| 
\end{align*}

We have that $\xi^{n+1}=P_{B(0,1)}(\xi^{n}+\delta \tau^{n+1} \gradh u^{n})$ hence by definition of the projection,
\[\forall \bar{\xi} \in B(0,1) \qquad \langle\xi^{n+1}-(\xi^{n}+\delta \tau^{n+1} \gradh u^n),\bar{\xi}-\xi^{n+1}\rangle \geq 0\]

This gives us
\[\langle \gradh u^n,\bar{\xi}-\xi^{n+1}\rangle\leq \langle \partial_t \xi^n,\bar{\xi}-\xi^n \rangle\]
Combining this with $\langle \partial_t \xi^n,\bar{\xi}\rangle-\langle \partial_t \xi^n,\xi^n\rangle \leq 2\sqrt{N\times M}|\partial_t \xi^n|$ (which holds by Cauchy-Schwarz's inequality, $|\bar{\xi}|_\infty \leq 1$ and $|\xi^n|_\infty \leq 1$), we find that
\[|u^n-\bar{u}|^2\leq 2\sqrt{N\times M}|\partial_t \xi^n|+|\partial_t u^n| |u^n-\bar{u}| \]
The announced inequality  easily follows.


\end{document}